\documentclass[11pt]{article}
\usepackage{amsmath,amsthm,amsfonts,amssymb,bm,wasysym}
\usepackage{subfigure}
\usepackage{graphicx}
\usepackage[usenames]{color}
\usepackage{verbatim}
\parskip \medskipamount
\newtheorem{theorem}{Theorem}[section]
\newtheorem{definition}[theorem]{Definition}

\numberwithin{equation}{section}
\newtheorem{lemma}[theorem]{Lemma}
\newtheorem{proposition}[theorem]{Proposition}
\newtheorem{remark}[theorem]{Remark}

\newtheorem{claim}[theorem]{Claim}

\newtheorem{conjecture}[theorem]{Conjecture}
\numberwithin{equation}{section}

\def\N{\mathbb{N}}
\def\Z{\mathbb{Z}}

\def\R{\mathbb{R}}

\def\S{\mathbb{S}}

\def\F{\mathcal{F}}

\def\B{\mathcal{B}}

\def\tr{{\rm{tr}}}
\renewcommand{\phi}{\varphi}
\renewcommand{\epsilon}{\varepsilon}

\allowdisplaybreaks

\newcommand{\1}{{\text{\Large $\mathfrak 1$}}}

\newcommand{\cov}{\operatorname{Cov}}

\newcommand{\til}{\widetilde}

\newcommand{\pr}[1]{\mathbb{P}\!\left(#1\right)}
\newcommand{\E}[1]{\mathbb{E}\!\left[#1\right]}

\newcommand{\prcond}[3]{\mathbb{P}_{#3}\!\left(#1\;\middle\vert\;#2\right)}
\newcommand{\econd}[2]{\mathbb{E}\!\left[#1\;\middle\vert\;#2\right]}

\newcommand{\norm}[1]{\left\| #1 \right\|}
\newcommand{\fract}[2]{{\textstyle\frac{#1}{#2}}}
\newcommand{\tn}{|\kern-.1em|\kern-0.1em|}

\newcommand\be{\begin{equation}}
\newcommand\ee{\end{equation}}

\def\eps{\varepsilon}

\begin{document}

\title{\bf Self-interacting random walks}

\author{
Yuval Peres\thanks{Microsoft Research, Redmond, Washington, USA; peres@microsoft.com} 
\and Serguei Popov\thanks{University of Campinas, Campinas SP, Brazil; popov@ime.unicamp.br}
\and Perla Sousi\thanks{University of Cambridge, Cambridge, UK;   p.sousi@statslab.cam.ac.uk}
}
%\date{}
\maketitle
\thispagestyle{empty}

\begin{abstract} 

Let $\mu_1,\ldots, \mu_k$ be $d$-dimensional probability measures in $\R^d$ with mean $0$. At each step we choose one of the measures based on the history of the process and take a step according to that measure. 
We give conditions for transience of such processes and also construct examples of recurrent processes of this type.
In particular, in dimension~$3$ we give the complete picture: every walk generated by two measures is transient and there exists a recurrent walk generated by three measures.
%The following random process in $\R^3$ is studied. Let $\mu_1,\mu_2$ be $2$ measures in $\R^3$ whose support spans $\R^3$.  We prove that a random walk with the property that the law of $X_{n+1}-X_n$ given $\F_n$ is one of the measures $\mu_1,\mu_{2}$ is transient in $3$ dimensions. This is in fact the maximum number of walks that combined in any adapted way result in transience in $3$ dimensions. This answers positively the question of Benjamini et al on whether $2$ fully supported measures in $4$ dimensions always generate a transient walk. In this paper, we also construct recurrent random walks in $d\geq 3$ dimensions that are generated using $d$ fully supported measures in $\R^d$. Finally we give a sufficient condition on the covariance matrices of the measures used to generate a random walk in order to get transience.
\newline
\newline
\emph{Keywords and phrases.} Transience, recurrence, Lyapunov function.
\newline
MSC 2010 \emph{subject classifications.}
Primary 60G50; % Random walks
Secondary 60J10. 
\end{abstract}

\section{Introduction}
%In this paper we study self interacting random walks in dimensions $3$ and above and prove recurrence and transience results. 

Let $\mu_1$ and $\mu_2$ be two zero mean measures in $\R^4$ 
with finite supports that span the whole space. On the first visit to a site the jump of the process has law $\mu_1$ and at further visits it has law $\mu_2$.
The following question was posed in~\cite{BenKozScha}:  Is the resulting walk transient? 

More generally, one can consider any adapted rule (i.e., a rule depending on the history of the process) for choosing between $\mu_1$ and $\mu_2$, and ask
the same question. It turns out that the answer to this question is positive, even in~$\R^3$, as proved in Theorem~\ref{thm:ddimd-1meas} below.
Moreover, in $3$ dimensions this result is sharp, in the sense that one can construct an example of a recurrent walk with three measures, as shown in Theorem~\ref{thm:ddimrecurrence}.

This naturally fits into the wider context of random walks that are not Markovian, namely where the next step the walk takes also depends on the past.
Recently there has been a lot of interest in random walks of this kind.
% that are not Markovian, namely where the next step the walk takes also depends on the past. 
A large class of such walks are the so-called vertex (or edge) reinforced random walks, where the walker chooses the next vertex to jump to with weight proportional to the number of visits to that vertex up to that time; see e.g.~\cite{Benaim, MerkRolSil, PemVolkov, RaimSchap, Volkov}.
Another class of such walks is the so-called excited random walks, when 
the transition probabilities depend on whether it is the first visit to a site or not, see e.g.~\cite{BenWilson, BerRam, MenPopRamVach, HofHol, Zerner}.

In this paper we study transience and recurrence for walks in dimensions $3$ and above that are generated by a finite collection of step distributions. We now give the precise definition of the walks we will be considering. 

\begin{definition}\label{def:randomwalk}\rm{
Let $\mu_1,\ldots,\mu_k$ be $k$ probability measures in $\R^d$ and for each $j=1,\ldots,k$, let $\xi^j_1,\xi^j_2,\ldots$
be i.i.d.\ with law $\mu_j$. Define an \textit{adapted rule} $\ell=(\ell(i))_i$ with respect to a filtration $(\F_i)$ to be a process such that $\ell(i) \in \{1,\ldots,k\}$ and is $\F_i$ measurable for all $i$.
We will say that the  walk~$X$, with $X_0=0$, is generated by the measures $\mu_1,\ldots, \mu_k$ and the rule $\ell$ if 
\[
X_{i+1} = X_i + \xi^{\ell(i)}_{i+1}.
\]
}
\end{definition}

We say that a measure $\mu$ in $\R^d$ has mean $0$ if $\int_{\R^d} x \mu(dx) =0$. Also we write that a measure $\mu$ has $\beta$ moments, if 
$\E{\norm{Z}^{\beta}}<\infty$, where $Z \sim \mu$.
We define the covariance matrix of $\mu$ as follows: $\cov(\mu) = \left(\E{Z_i Z_j}\right)_{i,j=1}^{d}$.

Note that if $\mu$ is a measure in $\R^d$, then it has an invertible covariance matrix if and only if its support contains $d$ linearly independent vectors of $\R^d$. We will call such measures $d$-dimensional.

In this paper we are mainly interested in the following two questions: 
\begin{itemize}
\item Let $\mu_1,\ldots,\mu_k$ be mean $0$ probability measures in $\R^d$.
What are the conditions on the measures so that for every adapted rule $\ell$ the resulting walk is transient?

\item For a given dimension $d$, how do we construct examples of recurrent walks generated by $k$ $d$-dimensional mean $0$ measures? How small can this number $k$ be made?
\end{itemize}

In Section~\ref{sec:transcond} we state our results concerning the first question and in Section~\ref{sec:recurrenceintro} about the second one. Observe that Theorems~\ref{thm:ddimd-1meas} and~\ref{thm:ddimrecurrence} give a complete picture in dimension $3$: any two mean $0$ measures with $2+\beta$ moments, for some $\beta>0$, always generate a transient walk, while there is an example of a recurrent walk generated by three $3$-dimensional measures of mean $0$ with a suitable adapted rule.

%In particular, this answers positively the following question posed in~\cite{BenKozScha}: 
%let $\mu_1$ and $\mu_2$ be two  measures in $\R^4$ 
%whose support spans the whole space. On the first visit to a site the jump of the process has 
%law $\mu_1$ and at further visits it has law $\mu_2$. Is the resulting walk transient? 

\subsection{Conditions for transience}\label{sec:transcond}

\begin{theorem}\label{thm:ddimd-1meas}
Let $\mu_1,\mu_{2}$ be $d$-dimensional measures in $\R^d$, $d\geq 3$, with zero mean and $2+\beta$ moments, for some $\beta>0$. If $X$ is a random walk generated by these measures and an arbitrary adapted rule $\ell$, then $X$ is transient.
\end{theorem}

The following result will be used in the proof of Theorem~\ref{thm:ddimd-1meas} but is also of independent interest, since it gives a sufficient condition on the covariance matrices of the measures used in order to generate a transient random walk $X$ for an arbitrary adapted rule $\ell$. 

For a matrix $A$ we write $A^T$ for its transpose, $\lambda_{\max}(A)$ for its maximum eigenvalue and $\tr(A)$ for its trace.

\begin{theorem}\label{thm:kmeasddim}
Let $\mu_1,\ldots, \mu_k$ be mean $0$ measures in $\R^d$, $d\geq 3$, with $2+\beta$ moments, for some $\beta>0$. Suppose that there exists a matrix $A$ such that for all $i$ we have
\begin{align}\label{eq:tracecondition}
\tr(A M_i A^T) > 2 \lambda_{\max}(A M_i A^T),
\end{align}
where $M_i$ is the covariance matrix of the measure $\mu_i$. If $X$ is a random walk generated by these measures and an arbitrary adapted rule $\ell$, then $X$ is transient.
\end{theorem}
We will refer to~\eqref{eq:tracecondition} as the \textit{trace condition}.

It turns out that the local central limit theorem implies the following lower bound on the number of measures needed to generate a transient walk.
\begin{proposition}\label{pro:localclt}
Let $\mu_1, \ldots, \mu_k$ be mean $0$ measures in $d \geq 2k+1$ with $2+\beta$ moments, for some $\beta>0$. Then the random walk $X$ generated by these measures and an arbitrary adapted rule $\ell$ is transient.
\end{proposition}

We will prove Proposition~\ref{pro:localclt} in the beginning of  Section~\ref{sec:proofoftheorem} and then Theorems~\ref{thm:kmeasddim} and~\ref{thm:ddimd-1meas} in Sections~\ref{sec:tracetransience} and~\ref{sec:examples} respectively. Then in Proposition~\ref{pro:tracediagonal} in Section~\ref{sec:diagonal} we discuss the case when the covariance matrices are jointly diagonalizable.
We present a conjectured sufficient condition for transience at the end of the paper.

%
%The following corollary follows immediately from Theorem~\ref{thm:ddimd-1meas}. In this simple case, though, where there are  $2$ measures in $3$ dimensions we give an explicit construction of the matrix $A$, whose existence is guaranteed by Proposition~\ref{pro:trace}, in Section~\ref{sec:examples}.

%\begin{corollary}\label{cor:2meas3dim}
%Let $\mu_1,\mu_{2}$ be 2 measures in $\R^d$ with $0$ mean and such that 
%\[
%\max_{j\leq k} \E{\norm{\xi_1^j}^{2+\beta}}\leq M<\infty,
%\]
%for some $\beta>0$ and with covariance matrices that are positive definite. If $X$ is the random walk generated by these two measures, then $X$ is transient.
%\end{corollary}

\subsection{Recurrence}\label{sec:recurrenceintro}

We now define a random walk in $d$ dimensions, which is generated by $d$ measures that are fully supported in $\R^d$ and we will prove that it is recurrent.
\newline
Let $e_0,\ldots,e_{d-1}$ be the coordinate vectors 
in~$\Z^d$.
We consider a random walk $(X_n,n=0,1,2,\ldots)$ on $\Z^d, 
d\geq 3$, defined in the following way. Fix a parameter $\gamma>0$, and
for $x=(x_0,\ldots,x_{d-1})\in\Z^d$ define
$\varrho(x)=\min\{k:|x_k|=\max_{j=0,\ldots,d-1}|x_j|\}$. Then
\[
X_{n+1} = X_n + \xi_{n+1},
\]
where $\xi_{n+1} = \pm e_{\varrho(X_n)}$ with probabilities $\frac{\gamma}{2(\gamma+d-1)}$ and $\xi_{n+1} = \pm e_k$ for $k \neq \varrho(X_n)$ with probabilities $\frac{1}{2(\gamma+d-1)}$. In words, we choose the maximal (in absolute value) coordinate of~$X_n$
with weight~$\gamma$ and all the other coordinates with weight~$1$, 
and then add~$1$ or~$-1$ to the chosen coordinate with equal
probabilities.

\begin{theorem}\label{thm:ddimrecurrence}
For each $d\geq 3$ there exists large enough $\gamma_d$ such 
that the random walk~$X$ is recurrent for all $\gamma\geq \gamma_d$.
\end{theorem}

We will prove Theorem~\ref{thm:ddimrecurrence} in Section~\ref{sec:recurrence}. The proof of this result relies on the explicit construction of a suitable Lyapunov function, but it is rather involved, so in Section~\ref{sec:recurrence} we also give simpler examples of a finite number of $d$-dimensional measures and adapted rules that generate a recurrent walk in $d$ dimensions.

\section{Proofs of transience}\label{sec:proofoftheorem}

In this section we give the proofs of the results on transience. We first prove Proposition~\ref{pro:localclt}, since its proof is short and elementary.

\begin{proof}[{\bf Proof of Proposition~\ref{pro:localclt}}]

In order to prove this proposition, let us first give an equivalent definition of the random walk that we are considering. 

For each $j=1,\ldots,k$, let $\zeta^j_1,\zeta^j_2,\ldots$
be i.i.d.\ with law $\mu_j$. 
For an adapted rule $\ell$ we define for all $j \in \{1,\ldots,k\}$ 
\[
r(j,i) = \sum_{m=1}^{i}\1(\ell(m)=j)
\]
and then writing $\hat{r}_i = r(\ell(i),i)+1$ we let
\[
X_{i+1} = X_i + \zeta^{\ell(i)}_{\hat{r}_i}.
\]
It is easy to see by induction that the process $X$ has the same law as the process of Definition~\ref{def:randomwalk}.
\newline
Let $R>0$ and for every $n$ we define the event
\[
A_n = \left\{\exists \ i_1,\ldots, i_k\geq 0: i_1 + \ldots + i_k = n \text{ and } \sum_{j=1}^{k} \sum_{\ell=1}^{i_j}\zeta_i^j \in \B(0,R)  \right\}.
\]
We now fix a choice of $i_1,\ldots, i_k$ such that $i_1+\ldots + i_k =n$. Then by~\cite[Corollary/Theorem~6.2]{Esseen} we get for a positive constant $c$
\[
\pr{\sum_{j=1}^{k} \sum_{\ell=1}^{i_j}\zeta_i^j \in \B(0,R)} \leq \frac{cR^d}{n^{d/2}},
\]
since there must exist some $i_j$ which is at least $n/k$. It is easy to see that the total number of $k$-tuples $(i_1,\ldots,i_k)$ with $i_j\geq 0$ for all $j$ and $\sum_j i_j =n$ is equal to $\binom{n-1}{k-1}$. Since 
$\binom{n+k-1}{k-1} \leq c_1 n^{k-1}$, for a positive constant $c_1$, we deduce that
\[
\pr{A_n} \leq c'R^d \frac{n^{k-1}}{n^{d/2}} = \frac{c'R^d}{n^{d/2-k+1}},
\]
which is summable if $d \geq 2k+1$. Hence, from Borel-Cantelli we obtain that a.s.\ only finitely many of the events $A_n$ happen. 
\newline
Now notice that for every $n$ we have
\[
\{X_n \in \B(0,R)\} \subseteq A_n,
\]
and hence we deduce that a.s.\ for all sufficiently large $n$, the random walk at time $n$ will stay outside of the ball $\B(0,R)$. Since this is true for any $R>0$, we get that if $d \geq 2k+1$ the random walk is transient.
\end{proof}

\subsection{Trace condition and transience}\label{sec:tracetransience}

In this section we give the proof of Theorem~\ref{thm:kmeasddim}. First we state and prove some preliminary results. 

The following lemma is a standard result, but we state and prove it here for the sake of completeness.
\begin{lemma}\label{lem:supermg}
Let $(S_t)$ be a random walk generated by $k$ zero mean measures and an arbitrary adapted rule $\ell$. Let $\F_t = \sigma(S_0,\ldots,S_t)$ be its natural filtration. Let $\alpha, r_0>0$ and define $\phi(x) = \norm{x}^{-\alpha} \wedge r_0^{-\alpha}$. If the process $(\phi(S_t))$ is a super-martingale, then $S$ is transient, in the sense that a.s.
\[
\norm{S_t} \to \infty \ \text{ as } \ t\to \infty.
\]
\end{lemma}

\begin{proof}[{\bf Proof}]
We first show that a.s.\ 
\begin{align}\label{eq:limsupas}
\limsup_{t\to \infty} \norm{S_t} = \infty.
\end{align}
Indeed, there exist $u \in \S^{d-1}$, $\epsilon>0$ and $h>0$ such that for all $j\in \{1,\ldots,k\}$
\[
\pr{\langle Z_j, u\rangle >\epsilon} \geq h, 
\]
where $Z_j \sim \mu_j$.
This implies that for all $m,n\in \N$ we have 
\[
\prcond{\langle S_{n+m} - S_n, u\rangle >\epsilon m}{\F_n}{} \geq h^m.
\]
Hence this shows that a.s.\ $\limsup_t |\langle S_t, u\rangle | \geq \epsilon m/2$ for all $m$, and so~\eqref{eq:limsupas} holds.
Clearly, this implies that a.s.\ 
\begin{align}\label{eq:liminf}
\liminf_{t\to \infty} \phi(S_t) = 0.
\end{align}
Since $(\phi(S_t))_t$ is a positive super-martingale, the a.s.\ super-martingale convergence theorem gives that $\lim_{t\to \infty}\phi(S_t)$ exists a.s.\ and thus from~\eqref{eq:liminf} we deduce that a.s.\ $\lim_{t\to \infty} \phi(S_t) = 0$,
which means that a.s.\ $\norm{S_t} \to \infty$ as $t\to \infty$.
\end{proof}

%Let $r_0<r<\norm{x}<R$.
%We can now apply the Optional Stopping Theorem to get
%\[
%\phi(x) \geq \estart{\phi(S_{T_r \wedge T_R})}{x} = \frac{1}{r^\alpha} \prstart{T_r < T_R}{x} + \frac{1}{R^\alpha} (1-\prstart{T_r < T_R}{x}),
%\]
%and hence
%\[
%\prstart{T_r < T_R}{x} \leq \frac{\frac{1}{\|x\|^\alpha} - \frac{1}{R^\alpha}}{\frac{1}{r^\alpha} - \frac{1}{R^\alpha}}.
%\]
%Thus letting $R \to \infty$, we obtain $\prstart{T_r < \infty}{x} \leq \frac{r^\alpha}{\|x\|^\alpha}$.
%
% We will now show that a.s.\ $\|S_t\| \to \infty$, as $t \to \infty$. 
% We define $T_{n^\beta} = \inf\{t \geq 0: \|S_t\| = n^\beta \}$ for $\beta = 2+2/\alpha$. The same argument as above shows that for all $x$ with $\norm{x}<n^\beta$ we have $T_{n^\beta} <\infty$ $\mathbb{P}_x$-a.s. 
% For every $n$ we define the event
%\[
%A_n = \{\|S_t \|>n, \forall t \geq T_{n^{\beta}} \}.
%\]
%Then using the fact that  $T_{n^{\beta}}< \infty$ $\mathbb{P}_0$-a.s.\ we obtain that for $n>r_0$
%\[
%\prstart{A_n^c}{0} = \estart{\prstart{T_n< \infty}{S_{T_{n^\beta}}}}{0} \leq \frac{n^{\alpha}}{n^{\beta\alpha}} = \frac{1}{n^{\alpha +1}}.
%\]
%Since $\alpha >0$, we get that $\sum_n \prstart{A_n^c}{0} < \infty$, and hence, applying Borel-Cantelli we get that only finitely many of the events $A_n^c$ happen. Therefore a.s.\ $\|S_t\|$ diverges as $t \to \infty$.

The following lemma shows that if the covariance matrices of the measures used to generate the walk~$X$ satisfy the trace condition~\eqref{eq:tracecondition}, then there is a function $\phi$ such that $\phi(X)$ is a super-martingale. 

\begin{lemma}\label{lem:taylorexpansion}
Let $\phi(x)=\norm{x}^{-\alpha}\wedge 1$, for $x\in \R^d$. Let $\mu_1,\ldots,\mu_k$ be zero mean measures in $\R^d$ with $2+\beta$ moments, for some $\beta >0$, and with covariance matrices $M_1,\ldots,M_k$ satisfying for all $i=1,\ldots,k$
\[
\tr(M_i) > 2\lambda_{\max}(M_i).
\]
There exists $\alpha>0$ small enough and a constant $r_0$ so that if $\norm{x} \geq r_0$, then
for all $i=1,\ldots, k$ if $Z_i \sim \mu_i$ 
\begin{align}\label{eq:phiequation}
\E{\phi(x + Z_i) - \phi(x)} \leq 0.
\end{align}
\end{lemma}

\begin{proof}[{\bf Proof}]

It suffices to prove~\eqref{eq:phiequation} for a fixed $i$. Since the covariance matrix $M_i$ is positive definite, there is  an orthogonal matrix $U$ such that $U M_i U^T$ is diagonal with non-negative eigenvalues. The matrix $U M_i U^T$ is the covariance matrix of the random variable $UZ_i$.
\newline
Since $U$ is orthogonal, we get that for all $x$
\begin{align}\label{eq:orthogonal}
\phi(U(x+Z_i)) = \phi(x+Z_i) \ \text{ and } \ \phi(Ux) = \phi(x).
\end{align}
In order to prove the lemma, we will apply Taylor expansion up to second order terms to the function $\phi$ around $Ux$ evaluated at $UZ_i$. We will drop the dependence on $i$ from $UZ_i$ and write simply $Z$ and $x$ instead of $UZ$ and $Ux$ in view of~\eqref{eq:orthogonal} to lighten the notation. 

So, let $Z$ have covariance matrix $M$ which is in diagonal form and with diagonal elements $\lambda_1,\ldots, \lambda_d$. Let $\til{Z}= Z \1(\|Z\| \leq \|x\|/2)$. Note that if a.s.\ $\norm{Z} \leq B$ for a positive constant $B$, then $\til{Z} = Z$ if $\norm{x}\geq 2B$. The calculations below are a bit simpler in this case, since $\til{Z}$ would have mean $0$ and the same covariance matrix as $Z$.

If $\norm{x}\geq 2$, then $\norm{x+\til{Z}}\geq 1$ and so $\phi(x+\til{Z}) = \norm{x+\til{Z}}^{-\alpha}$. In what follows we abbreviate 
\[
\phi'_i(x) = \frac{\partial \phi(x)}{\partial x_i}, \quad \phi''_{ij}(x) = \frac{\partial^2 \phi(x)}{\partial x_i\partial x_j}, \quad \phi'''_{ijk}(x) = \frac{\partial^3 \phi(x)}{\partial x_i\partial x_j \partial x_k}.
\]
Applying Taylor expansion to $\phi$ up to second order terms gives for some $\eta \in (0,1)$
\begin{align*}
\phi(x+\til{Z}) &= \phi(x) + \langle \nabla \phi(x), \til{Z} \rangle + \frac{1}{2} \sum_{i,j=1}^{d}  \phi''_{ij}(x) \til{Z}_i \til{Z}_j +\frac{1}{3!} \sum_{i,j,k=1}^{d}  \phi'''_{ijk}(x+\eta \til{Z}) \til{Z}_i \til{Z}_j \til{Z}_k \\
&= \phi(x) + \langle \nabla \phi(x), \til{Z} \rangle + \frac{1}{2} \sum_{i,j=1}^{d}  \phi''_{ij}(x) {Z}_i {Z}_j 
 +\frac{1}{3!} \sum_{i,j,k=1}^{d}  \phi'''_{ijk}(x+\eta \til{Z}) \til{Z}_i \til{Z}_j \til{Z}_k \\
 &\quad - 
\sum_{i,j=1}^{d}  \phi''_{ij}(x) {Z}_i {Z}_j \1\left(\norm{Z}\geq \frac{\norm{x}}{2}\right).
\end{align*}

\begin{claim}\label{cl:estimates}
There exist positive constants $C,C_1$ such that for all $i,j$
\[
\left|\E{\langle \nabla \phi(x), \til{Z} \rangle } \right| \leq \frac{C}{\|x\|^{\alpha+\beta +2}} \ \text{ and } \ \E{Z_i Z_j \1(\|Z\|\geq \|x\|/2)} \leq \frac{C_1}{\norm{x}^\beta}.
\]
\end{claim}

\begin{proof}[{\bf Proof}]
By H{\"o}lder's inequality we have
\[
\E{\|Z\| \1(\|Z\|\geq \|x\|/2)} \leq \frac{2^{\beta +1}\E{\|Z\|^{\beta +2}}}{\|x\|^{\beta +1}} \leq \frac{K}{\|x\|^{\beta+1}}.
\]
Since $\E{Z}= 0$, we have $\E{\til{Z}} = \E{\til{Z}  -Z}$, and hence 
\[
\left\|\E{\til{Z}}\right\| = \|\E{Z \1(\|Z\|\geq \|x\|/2)} \| \leq \E{\|Z\| \1(\|Z\|\geq \|x\|/2)} \leq  \frac{{K}}{\|x\|^{\beta+1}}.
\]
For the first term of the Taylor expansion we have for a positive constant $C$
\[
\left|\E{\langle \nabla \phi(x), \til{Z} \rangle } \right| = \sum_{i=1}^{d} \frac{\alpha |x_i|}{\|x\|^{\alpha+2}}\left| \E{\til{Z}_i} \right| \leq \sum_{i=1}^{d} \frac{\alpha |x_i|}{\|x\|^{\alpha+2}}\left\|\E{\til{Z}}\right\|\leq \frac{\alpha d K\|x\|}{\|x\|^{\alpha+\beta +3}}= \frac{C}{\|x\|^{\alpha+\beta +2}}.
\]
For all $i,j$ we have by H{\"o}lder's inequality again
\[
\E{Z_i Z_j \1(\|Z\|\geq \|x\|/2)} \leq \E{\norm{Z}^2 \1(\|Z\|\geq \|x\|/2)} \leq \frac{C_1}{\norm{x}^\beta},
\]
thus proving the claim.
\end{proof}

We continue proving Lemma~\ref{lem:taylorexpansion}.
For the second order terms we write
\[
\E{\til{Z}_i \til{Z}_j} = \E{Z_i Z_j \1(\|Z\|\leq \|x\|/2)} = \E{Z_i Z_j} - \E{Z_i Z_j \1(\|Z\|\geq \|x\|/2)},
\]
and hence since for $i\neq j$ we have $\E{Z_i Z_j} =0$, by Claim~\ref{cl:estimates} we get
\[
\left|\E{\til{Z}_i \til{Z}_j} \right| \leq \frac{C_1}{\|x\|^{\beta}} \ \text{ and } \ \left| \sum_{i=1}^{d} \phi''_{ii}(x) \E{Z_i^2 \1(\norm{Z} \geq \norm{x}/2)} \right| \leq \frac{C_2}{\norm{x}^{\alpha+\beta+2}}.
\]
Since for all $i$ we have $\E{Z_i^2} = \lambda_i$, we obtain
\begin{align}\label{eq:tracetaylor}
\sum_{i=1}^{d} \phi''_{ii}(x) \E{Z_i^2}  = \sum_{i=1}^{d} \lambda_i \frac{-\alpha\|x\|^2 + \alpha(\alpha+2) x_i^2}{\|x\|^{\alpha+4}} = \sum_{i=1}^{d}  \frac{\alpha x_i^2(\lambda_i(\alpha+2) - \sum_{j=1}^{d}\lambda_j)}{\|x\|^{\alpha+4}}.
\end{align}

The rest of the second order terms can be bounded as follows:
\begin{align*}
\left|\sum_{i\neq j}  \phi''_{ij}(x)  \E{\til{Z}_i \til{Z}_j} 
 \right| = \sum_{i\neq j}\frac{\alpha (\alpha+2)|x_i||x_j|}{\|x\|^{\alpha+4}}\left|\E{\til{Z}_i\til{Z}_j}\right| \leq \sum_{i\neq j}\frac{\alpha (\alpha+2)|x_i||x_j|}{\|x\|^{\alpha+4}} \frac{C_1}{\|x\|^{\beta}} \leq \frac{C_3}{\|x\|^{\alpha+\beta +2}}.
\end{align*}
For the remainder in the Taylor expansion we have
\[
\max_{i,j,k}\left| \phi'''_{ijk}(x+\eta \til{Z}) \right| \leq \frac{C}{\|x + \eta \til{Z}\|^{\alpha+3}} \leq \frac{C_4}{\|x\|^{\alpha+3}},
\]
since $\left\|\til{Z}\right\| \leq \|x\|/2$. 
We want to control $\E{{\phi}(x+Z) - {\phi}(x)}$. We write
\begin{align}\label{eq:phidiffer}
\E{{\phi}(x+Z) - {\phi}(x)} = \E{{\phi}(x+Z) - {\phi}(x+\til{Z})} + \E{{\phi}(x+\til{Z}) - {\phi}(x)}
\end{align}
and by Markov's inequality since $\E{\norm{Z}^{2+\beta}} <\infty$
\[
\E{\left|{\phi}(x+Z)- {\phi}(x+\til{Z})\right|} \leq \pr{\|Z\|\geq \|x\|/2} \leq \frac{C_5}{\|x\|^{\beta +2}}.
\]

Since $\beta>0$, if we take $0<\alpha <\beta$, then we obtain that there exists a constant $r_0>1$ so that  for $\norm{x}>r_0$
\begin{align}\label{eq:thistogether}
 \left| \E{\langle \nabla \phi(x), \til{Z} \rangle} \right| &+ \frac{1}{2}\left|\sum_{i,j=1}^{d}  \phi''_{ij}(x) \E{{Z}_i {Z}_j \1\left(\norm{Z}\geq \frac{\norm{x}}{2}\right)} \right| +
\frac{1}{3!} \sum_{i,j,k=1}^{d}\left|\E{ \phi'''_{ijk}(x+\eta \til{Z}) \til{Z}_i \til{Z}_j \til{Z}_k} \right| \nonumber\\
&+ \left|\E{{\phi}(x+\til{Z}) - {\phi}(x+Z)} \right| 
 \leq \left| \frac{1}{2} \sum_{i,j=1}^{d} \phi''_{ij}(x) \E{{Z}_i {Z}_j } \right|.
\end{align}
The assumption on the trace of the matrix $M$ gives that for $\alpha$ small enough (smaller than $\beta$) $\sum_{j=1}^{d} \lambda_j > \lambda_i(\alpha+2)$ for all $i$, and hence using~\eqref{eq:tracetaylor} we get for $\norm{x}\geq r_0$
\[
\sum_{i =1}^{d} \phi''_{ii}(x) \E{{Z}_i^2} < 0.
\]
This and the inequality~\eqref{eq:thistogether} finishes the proof.
\end{proof}

We now have all the required ingredients to give the proof of Theorem~\ref{thm:kmeasddim}.

\begin{proof}[{\bf Proof of Theorem~\ref{thm:kmeasddim}}]

Let $r_0>1$ be the constant of Lemma~\ref{lem:taylorexpansion}.
Let $\til{\phi}(x) = \|x\|^{-\alpha}\wedge r_0^{-\alpha}$, for $\alpha>0$ as in Lemma~\ref{lem:taylorexpansion}. Notice that when $\norm{x}\geq r_0$, then $\til{\phi}(x) = \phi(x) = \norm{x}^{-\alpha}$.
We will first show that if $Y_t=AX_t$, then 
\begin{align}\label{eq:supermgproperty}
\econd{\til{\phi}(Y_{t+1})}{\F_t} \leq \til{\phi}(Y_t).
\end{align}

Since $r_0>1$, we have $\til{\phi}(x) \leq \phi(x)$ for all $x$. 
So we get
\begin{align*}
\econd{\til{\phi}(Y_{t+1}) - \til{\phi}(Y_t)}{\F_t} &= \econd{(\til{\phi}(Y_{t+1}) - \til{\phi}(Y_t)) \1(\norm{Y_t} \geq r_0)}{\F_t} \\& \quad
+ \econd{(\til{\phi}(Y_{t+1}) - \til{\phi}(Y_t)) \1(\norm{Y_t} < r_0)}{\F_t} \\ &\leq \econd{(\phi(Y_{t+1}) - \phi(Y_t))\1(\norm{Y_t} \geq r_0)}{\F_t},
\end{align*}
since $\til{\phi}(Y_t) = r_0^{-\alpha}$ if $\norm{Y_t} <r_0$ and
$\til{\phi}(x) \leq r_0^{-\alpha}$ for all $x$.
Since the covariance matrices of the measures used to generate the walk $Y$ satisfy the trace condition~\eqref{eq:tracecondition}, Lemma~\ref{lem:taylorexpansion}
gives that 
\[
\econd{(\phi(Y_{t+1}) - \phi(Y_t))\1(\norm{Y_t} \geq r_0)}{\F_t} \leq 0
\]
and this completes the proof of~\eqref{eq:supermgproperty}.
Therefore by Lemma~\ref{lem:supermg} we get that a.s.
$\norm{A X_t} = \norm{Y_t} \to \infty$ as $t\to \infty$. 
Since for all $t$ we have $\norm{A X_t} \leq \norm{A} \norm{S_t}$ and $\norm{A}>0$, we deduce that a.s.
\[
\norm{X_t} \to \infty \ \text{ as } t\to \infty,
\]
which concludes the proof of the theorem.
\end{proof}

\subsection{Two measures in $3$ dimensions}\label{sec:examples}

In this section we give the proof of Theorem~\ref{thm:ddimrecurrence}. 

\begin{proposition}\label{pro:trace}
Let $M_1,M_{2}$ be $3\times 3$ invertible positive definite matrices. Then there exists a $3\times 3$ matrix $A$ such that 
\[
\tr(AM_i A^T) > 2 \lambda_{\max}(A M_i A^T) \ \forall  \ i=1,2.
\]
\end{proposition}

\begin{proof}[{\bf Proof}]
We prove Proposition~\ref{pro:trace} by constructing the matrix $A$ of Theorem~\ref{thm:kmeasddim} directly. 

Let $\mu_1,\mu_2$ have covariance matrices $M_1$ and $M_2$ respectively and $\xi_i \sim \mu_i$ for $i=1,2$. Since $M_1$ is positive definite, there exists an orthogonal matrix $U$ such that $U M_1 U^T$ is diagonal, i.e.
\[
U M_1 U^T= \left( \begin{array}{ccc}
a & 0 & 0 \\
0 & b & 0 \\
0 & 0 & c \end{array} \right),
\]
where $a,b,c>0$ are the eigenvalues of $M_1$.
If we now multiply the vector $U\xi_1$ by the matrix $D$ given by
\[
D = \left( \begin{array}{ccc}
\frac{1}{\sqrt{a}} & 0 & 0 \\
0 & \frac{1}{\sqrt{b}} & 0 \\
0 & 0 & \frac{1}{\sqrt{c}} \end{array} \right),
\]
then $\cov(DU\xi_1) = I$, where $I$ stands for the $3\times 3$ identity matrix. 

So far we have applied the matrix $DU$ to the vector $\xi_1$ and we have to apply the same transformation to the vector $\xi_2$. The vector $DU\xi_2$ will have covariance matrix $\til{M}_2$. Since it is positive definite, it can be diagonalised, so there exists an orthogonal matrix $V$ such that 
\[
V\til{M}_2 V^T = \left( \begin{array}{ccc}
\lambda_1 & 0 & 0 \\
0 & \lambda_2 & 0 \\
0 & 0 & \lambda_3 \end{array} \right),
\]
where $\lambda_1\geq \lambda_2\geq\lambda_3>0$ are the eigenvalues in decreasing order. Applying the same transformation to $DU\xi_1$ is not going to change its identity covariance matrix, since $V$ is orthogonal. 
\newline
The condition we want to satisfy is 
\[
\lambda_1 + \lambda_2 +\lambda_3 > 2\lambda_i,
\]
for all $i=1,2,3$. Since the eigenvalues are in decreasing order, it is clear that this inequality is always satisfied for $i=2,3$. Suppose that $\lambda_2+\lambda_3 \leq \lambda_1$. Multiplying $DU\xi_2$ by the matrix
\[
B = \left( \begin{array}{ccc}
\frac{\sqrt{\lambda_2}}{\sqrt{\lambda_1}} & 0 & 0 \\
0 & 1 & 0 \\
0 & 0 & 1 \end{array} \right)
\]
will give us a random vector with covariance matrix 
\[
\left( \begin{array}{ccccc}
\lambda_2 & 0 & 0 \\
0 & \lambda_2 & 0 \\
0 & 0 & \lambda_3 \end{array} \right),
\]
which clearly satisfies the trace condition~\eqref{eq:tracecondition}. Multiplying $VDU\xi_1$ by the same matrix will give us a vector with covariance matrix
\[
\left( \begin{array}{ccc}
\frac{{\lambda_2}}{{\lambda_1}} & 0 & 0 \\
0 & 1 & 0 \\
0 & 0 & 1 \end{array} \right)
\]
which satisfies the trace condition~\eqref{eq:tracecondition}, since $\lambda_2\leq \lambda_1$.
\end{proof}

\begin{proof}[{\bf Proof of Theorem~\ref{thm:ddimd-1meas}}]
By projection to the first three coordinates, it is clear that it suffices to prove the theorem in $3$ dimensions.
\newline
In $d=3$, the statement of the theorem follows from Theorem~\ref{thm:kmeasddim} and Proposition~\ref{pro:trace}.
\end{proof}

\begin{remark}\rm{
It can be seen from the proof of Proposition~\ref{pro:trace} that if the measures $\mu_1$ and $\mu_2$ are supported on any $3$ dimensional subspaces of $\R^d$, then a walk $X$ generated by these measures and an arbitrary adapted rule is transient.  
}
\end{remark}

\subsection{The diagonal case}\label{sec:diagonal}

In this section we consider a particular case when for some basis of $\R^d$ the covariance matrices are in diagonal form and invertible. In this setting we prove that a random walk generated by $d-1$ measures and an arbitrary rule $\ell$ is transient.

\begin{proposition}\label{pro:tracediagonal}
Let $d\geq 4$ and $\mu_1,\ldots,\mu_{d-1}$ be mean $0$ probability measures in $\R^d$ with $2+\beta$ moments, for some $\beta>0$. Let $M_1, \ldots, M_{d-1}$ be their covariance matrices and suppose that $M_i M_j = M_j M_i$ for all $i,j$. Then there exists a $d\times d$ matrix $A$ such that 
\[
\tr(AM_i A^T) > 2 \lambda_{\max}(A M_i A^T) \ \forall  \ i\leq d-1.
\]
Therefore, a random walk $X$ generated by the measures $(\mu_i)_{i=1}^{d-1}$ and an arbitrary adapted rule $\ell$ is transient.
\end{proposition}

Before giving the proof of Proposition~\ref{pro:tracediagonal} we prove the following:
\begin{claim}\label{cl:continuity}
Let $M_1,\ldots, M_k$ be $d\times d$ invertible diagonal matrices with positive entries on the diagonal. For $A \neq 0$ we define
\begin{align}\label{eq:defpsi}
\Psi(A) = \max_{1\leq j\leq k} \frac{\norm{A M_jA^T}}{\tr(AM_jA^T)}.
\end{align}
Then the minimum of $\Psi(A)$ exists among all diagonal matrices $A$ and the minimizing matrix $\til{A}$ is invertible.
\end{claim}

\begin{proof}[{\bf Proof}]
Since $M_j$ is an invertible positive definite matrix, we can write $M_j = B_j B_j^T = B_j^2$, where $B_j = B_j^T$ is an invertible matrix.

Since scaling $A$ does not change the ratio in~\eqref{eq:defpsi}, we may assume that $\|A\|=1$ and restrict attention to such matrices. It is easy to see that the set $S =\{ A \text{ diagonal }: \|A\| =1\}$ is compact and the function $f_j(A) = \|AM_j A^T\|$ is continuous on $S$.

Let $g_j(A) = \tr(AM_j A^T) = \tr(AB_jB_j^TA^T) = \tn AB_j\tn ^2$, where $\tn C\tn ^2 = \sum_{i,j=1}^{d} c_{i,j}^2$ and we used $\tr(CC^T) = \tn C\tn ^2$. 

Since $B_j$ is invertible, we have $A B_j \neq 0$ for $A\neq 0$, so $g_j$ does not vanish on $S$. Thus as $g_j$ is continuous on $S$, we conclude that 
\[
A\mapsto \max_{1\leq j \leq k} \frac{f_j(A)}{g_j(A)}
\]
is continuous on $S$ and hence has a minimum. 

Let $\til{A}$ be the minimizing matrix with diagonal elements $\lambda_1,\ldots, \lambda_d \geq 0$. We will show that $\til{A}$ is invertible.
Suppose the contrary and assume without loss of generality that $\lambda_d =0$.
\newline
We prove that if we replace $\lambda_d =0$ by a small $\epsilon>0$, then we get a matrix $\til{A}_\epsilon$ with $\Psi(\til{A}_\epsilon) < \Psi(\til{A})$. Let the diagonal elements of $M_i$ be $(a_j^i)_{j=1}^{d}$, which are all strictly positive. Then for the matrix $M_i$ we will have for $s$ such that $\norm{M_i} = a_s^i$
\[
\frac{\lambda_{\max}(\til{A}M_i\til{A})}{\tr(\til{A} M_i \til{A})} = \frac{\lambda_s a_s^i}{\sum_{j}\lambda_j a_j^i}. 
\]

If $\til{A}_\epsilon$ has the same elements as $\til{A}$ except for the $(d,d)$ element which is replaced by $\epsilon>0$ such that 
$\epsilon< \frac{\lambda_i a_j^i}{a_d^i}$
for all $i=1,\ldots, d-1$ and all $j=1,\ldots, d-1$, then 
\[
\tr(\til{A}_\epsilon M_i \til{A}_\epsilon) = \tr(\til{A} M_i \til{A}) + \epsilon a_d^i,
\]
while $\lambda_{\max}(\til{A}_\epsilon M_i \til{A}_\epsilon) = \lambda_{\max}(\til{A} M_i \til{A})$. 

Replacing each $0$ element of $\til{A}$ by a sufficiently small number gives a matrix with smaller value of~$\Psi$, which contradicts the choice of $\til{A}$. Hence this shows that $\til{A}$ is invertible. 
\end{proof}

\begin{proof}[{\bf Proof of Proposition~\ref{pro:tracediagonal}}]

Since $M_i M_j = M_j M_i$ for all $i,j$, it follows (see for instance \cite[Theorem~2.5.5]{HornJohnson}) that there is
one orthogonal matrix that diagonalizes all the matrices $M_i$. So from now on we suppose that the $M_i$'s are diagonal. 

Recall the definition of $\Psi$ from~\eqref{eq:defpsi}.
Let $\til{A}$ be the $d\times d$ invertible matrix that minimizes $\Psi$
among all diagonal matrices (recall Claim~\ref{cl:continuity}). 

Write $\til{M}_i = \til{A} M_i \til{A}^T$ and 
\[
J = \left\{ j\leq d-1: \frac{\|\til{M}_j\|}{\tr(\til{M}_j)} = \Psi(\til{A})\right\}.
\]
Since $\til{A}$ and $M_i$ are diagonal invertible matrices, it follows that $\til{M}_i$ is also a diagonal invertible matrix. 
For each $j \leq d-1$ we can find $v_j \in \R^d$ such that $\|v_j\|=1$ and $\til{M}_j v_j = \|\til{M}_j\| v_j$. Note that since $\til{M}_j$ is diagonal, it follows that $v_j$ can be chosen to be one of the standard basis vectors of $\R^d$. Let $w\in \R^d$ have $\|w\|=1$ and $w\perp \{v_1,\ldots, v_{d-1}\}$. Then $w$ will also be one of the standard basis vectors of $\R^d$.

Next, we separate two cases. 
\newline
Case 1: For some $j \in J$ there is $u_j \perp v_j$ with $\|u_j\|=1$ and $\til{M}_j u_j = \|\til{M}_j\| u_j$. In this case,
\begin{align}\label{eq:star}
\tr(\til{M}_j) > \langle \til{M}_j v_j , v_j \rangle + \langle \til{M}_j u_j, u_j\rangle = 2\|\til{M}_j\|,
\end{align}
where the strict inequality follows from the fact that $\til{M}_j$ is invertible. Hence in the case where $\norm{\til{M}_j}$ has multiplicity at least $2$, we are done.  

Case 2: For each $j\in J$ the leading eigenvalue $\|\til{M}_j\|$ of $\til{M}_j$ has multiplicity one. We will show that this case leads to a contradiction; that is we can find another matrix with smaller value of $\Psi$ contradicting the choice of $\til{A}$ as the minimizer.

Let $A_\epsilon$ be the $d\times d$ matrix such that $A_\epsilon w = (1+\epsilon)w$ and $A_\epsilon z = z$ for all $z\perp w$. Note that $A_\epsilon$ will also be diagonal, since $w$ is one of the standard basis vectors of $\R^d$. 

Let us denote by $\gamma_j$ the second largest eigenvalue of $\til{M}_j$. Then the assumption of case 2 implies that for each $j\in J$ we have $\gamma_j <\|\til{M}_j\|$ and $\|\til{M}_j y\| \leq \gamma_j \|y\|$ for all $y \perp v_j$.

Choose $\epsilon>0$ such that $(1+\epsilon)^2 \|\til{M}_i\| <\tr(\til{M}_i)\Psi(\til{A})$ for all $i \notin J$ and $(1+\epsilon)^2 \gamma_j <\|\til{M}_j\|$ for all $j \in J$. 

Note that since $A_\epsilon$ is diagonal, $A_\epsilon^T = A_\epsilon$ and $\til{A}A_\epsilon$ is diagonal satisfying
\[
\Psi(A_\epsilon \til{A}) = \max_{1\leq i\leq d-1} \frac{\|A_\epsilon \til{M}_i A_\epsilon\|}{\tr(A_\epsilon \til{M}_i A_\epsilon)}.
\]
By completing $\{w,v_j\}$ to an orthonormal basis $\{b_m\}|_{m=1}^{d}$ of $\R^d$ we see that for all $i\leq d-1$ 
\begin{align}\label{eq:twostars}
\tr(A_\epsilon \til{M}_i A_\epsilon) > \tr(\til{M}_i),
\end{align}
since $\tr(M) = \sum_{m=1}^{d} \langle M b_m, b_m \rangle$ for any matrix $M$ and any orthonormal basis. The strict inequality follows again from the fact that the matrix $A_\epsilon \til{M}_i A_\epsilon$ is invertible. Also 
\[
\|A_\epsilon \til{M}_i A_\epsilon\| \leq \|A_\epsilon\|^2 \|\til{M}_i\| = (1+\epsilon)^2 \|\til{M}_i\|
\]
and for $j\in J$ we have for all $y\perp v_j$
\[
\|A_\epsilon M_j A_\epsilon y \| \leq (1+\epsilon) \|M_j (A_\epsilon y)\| \leq (1+\epsilon)\gamma_j \|A_\epsilon y\| \leq (1+\epsilon)^2 \gamma_j \|y\|,
\]
since $A_\epsilon y\perp v_j$.

We conclude that $\Psi(A_\epsilon \til{A}) <\Psi(\til{A})$ by considering separately in the max defining $\Psi$ the indices $i\notin J$ and $i\in J$, and applying~\eqref{eq:twostars}. This contradicts the choice of $\til{A}$ as a minimizer and establishes that case 2 is impossible. 
\end{proof}

%If the measures are non lattice, then for all $x$ and $r$ we have that 
%\[
%\prstart{S_n \in \B(x,r)}{0} \leq \frac{c}{n^{d/2}},
%\] 
%where $\B(x,1)$ stands for the ball of center $x$ and radius $r$. So, arguing in the same way as in the lattice case, we deduce that the combined walk is transient. 

\section{More measures may yield a recurrent walk}\label{sec:recurrence}

In this section we prove that the random walk described in Section~\ref{sec:recurrenceintro} is recurrent. First we give the simpler example that was mentioned in the Introduction.

Let $\S^{d-1}$ be the $d$-dimensional unit sphere, i.e.\ $\S^{d-1} = \{x\in \R^d: \norm{x}=1 \}$.
Let $C_1, \ldots, C_k$ be caps that cover the surface of the sphere with the property that the angle between any two vectors from the origin to points on the same cap is strictly smaller than ${\pi}/{2}$. For every cap $C_i$, for $i=1,\ldots,k$, we write $m(C_i)$ for the vector joining $0$ to the center of the cap $C_i$. Then we choose $v_{i,1},\ldots,v_{i,d-1}$ to be $d-1$ orthogonal vectors on the hyperplane orthogonal to $m(C_i)$. 

For every $x \in \R^d$, we write $C(x)$ for the first cap in the above ordering such that the vector joining $0$ and $x$ intersects that cap. 

\begin{theorem}\label{thm:moremeasures}
Let $X$ be a walk in $\R^d$ that moves as follows. When at $x$ it moves along the direction of $m(C(x))$ either $+1$ or $-1$ each with probability $1/2$ and along each of the other $d-1$ directions, i.e.\ along the vectors $v_{i(x),1},\ldots,v_{i(x),d-1}$ it moves independently as follows: $\pm 1$ with probabilities $\epsilon/2$ and stays in place with the remaining probability. Then $X$ is a recurrent walk, i.e.\ there is a compact set that is visited by $X$ infinitely many times a.s.
\end{theorem}

\begin{remark}\rm{
It can be shown that the ratio of the area of the unit sphere to the area of a cap as defined above with angle $\pi/2$ is equal to $2/I_{1/2}\left(\frac{d-1}{2},\frac{1}{2}\right)$, where $I$ is the regularized incomplete beta function. It is then elementary to obtain that the last quantity can be bounded below by $2^{d/2+1} >d$, so that in the above theorem at least $2^{d/2+1}$ measures are needed.
}
\end{remark}

\begin{proof}[{\bf Proof of Theorem~\ref{thm:moremeasures}}]
We define $\phi(x) = \log \|x\|$, for $x \in \R^d$. Then by Taylor expansion to second order terms we obtain for some $\eta \in (0,1)$
\[
\phi(x+Z) = \phi(x) + \langle \nabla \phi(x), Z \rangle + \frac{1}{2} \sum_{i,j}  \frac{\partial^2 \phi(x)}{\partial x_i \partial x_j}
 Z_i Z_j +
\frac{1}{3!} \sum_{i,j,k=1}^{d} \frac{\partial^3 \phi(x+\eta {Z})}{\partial x_i \partial x_j \partial x_k} {Z}_i {Z}_j {Z}_k.
\]
For each $i$ and positive constants $C,C_1$, since $Z$ is bounded, we have
\[
\frac{\partial \phi}{\partial x_i} = \frac{x_i}{\|x\|^2}, \quad 
\frac{\partial^2 \phi}{\partial x_i^2} = \frac{\sum_{j\neq i}x_j^2 - x_i^2}{\|x\|^4} \ \text{ and } \ \max_{i,j,k}\left|\frac{\partial^3 \phi(x+\eta {Z})}{\partial x_i \partial x_j \partial x_k} \right| \leq \frac{C_1}{\|x+\eta Z\|^{3}} \leq \frac{C}{\|x\|^{3}}.
\]
Let $u_1, \ldots, u_d$ be the vectors (basis of $\R^d$) as defined in the theorem. We now write both $x$ and $Z$ in this basis, i.e. we have that $x = \sum_{i=1}^{d}x_i u_i$ and $Z= \sum_{i=1}^{d}Z_i u_i$. Then for $i \neq j$, by independence, we get that $\E{Z_i Z_j} = 0$, while $\E{Z_1^2} = 1$ and for all $i>1$ we have that $\E{Z_i^2} = \epsilon$. Hence, putting all things together we obtain that 
\begin{align}\label{eq:second}
\sum_{i,j} \frac{\partial^2 \phi(x)}{\partial x_i \partial x_j}
\E{Z_i Z_j} =  \frac{(1+\epsilon(d-3))\|x\|^2 + 2(\epsilon-1)x_1^2}{\|x\|^4}.
\end{align}
For the first coordinate $x_1$ of $x$, when decomposed in the basis described above, we have that 
\[
x_1 = \|x\| \cos \theta,
\]
where $\theta$ is strictly smaller than $\pi/4$, 
so there exists $\delta>0$ so that 
$\cos \theta \geq (1+\delta)\sqrt{2}/2$. 
Hence, we can now bound \eqref{eq:second} from above by 
\[
\frac{x_1^2}{\norm{x}^4}\left(2(\epsilon-1) + \frac{1}{2(1+\delta)^2} (1+  \epsilon (d-3))\right),
\]
which can be made negative by choosing $\epsilon$ small enough. 
Notice that in absolute value the last expression is at least $c \norm{x}^{-2}$ for a positive constant $c$, and hence since $Z$ has mean $0$, it follows that for $\norm{x}$ large enough we have
\begin{align*}
 \left|\frac{1}{3!} \sum_{i,j,k=1}^{d} \E{\frac{\partial^3 \phi(x+\eta {Z})}{\partial x_i \partial x_j \partial x_k} {Z}_i {Z}_j {Z}_k} \right| 
\leq  \frac{1}{2} \left|\sum_{i,j=1}^{d} \frac{\partial^2 \phi(x)}{\partial x_i \partial x_j} \E{{Z}_i {Z}_j }  \right|.
\end{align*}
Therefore we deduce that for $\norm{x} \geq r_0$
\begin{align}\label{eq:goalproved}
\E{\phi(x+Z) - \phi(x)} \leq 0.
\end{align}
We now show that this implies recurrence. 
By the same argument used to show~\eqref{eq:limsupas} we get that a.s.
\[
\limsup_{t\to \infty} \norm{X_t} = \infty.
\]
Let $T_{r_0} = \inf\{t\geq 0: X_t \in \B(0,r_0) \}$.
By~\eqref{eq:goalproved} we obtain that $\phi(X_{t\wedge T_{r_0}})$ is a positive super-martingale.  Hence the a.s.\ martingale convergence theorem gives that $\lim_{t\to \infty} \phi(X_{t\wedge T_{r_0}})=Y$ exists a.s.\ and is finite. 
If $T_{r_0}=\infty$ with positive probability, then since $\phi(x) \to \infty$ as $x \to \infty$, then
$\lim \phi(X_{t\wedge T_{r_0}}) = \infty$ with positive probability, which is a contradiction. Therefore, $T_{r_0} <\infty$ a.s.
%This means that $\phi(X_{t\wedge T_{r_0}})$ is bounded a.s. Since $\phi(x) \to \infty$ as $x \to \infty$, then $\norm{X_{t\wedge T_{r_0}}}$  is also bounded, and because it converges, it follows that $T_{r_0}<\infty$ a.s. 
\end{proof}

We will now give the proof of Theorem~\ref{thm:ddimrecurrence}.

\begin{proof}[{\bf Proof of Theorem~\ref{thm:ddimrecurrence}}]
By~\cite[Theorem 2.2.1]{FayMalMensh} or analogously to the last part of the proof of Theorem~\ref{thm:moremeasures},
to prove recurrence it is enough to find
a nonnegative function~$f$ such that $f(x)\to\infty$ as $x\to\infty$,
and
\begin{equation}
\label{supermart}
\econd{f(X_{n+1})-f(X_n)}{X_n=x} \leq 0 \qquad \text{for all 
large enough $x$.}
\end{equation}

Before presenting the explicit construction of such a function,
let us informally explain the intuition behind this construction.
First of all, a straightforward computation shows that, if~$Y$ is a simple random walk in~$\Z^d$, then
\begin{align*}
 \mathbb{E}[\|Y_{n+1}\|-\|Y_n\| \mid Y_n=x] &=
\frac{d-1}{2d}\frac{1}{\|x\|} + O(\|x\|^{-2}),\\
\mathbb{E}[(\|Y_{n+1}\|-\|Y_n\|)^2 \mid Y_n=x] &=\frac{1}{d} + O(\|x\|^{-1}) .
\end{align*}
One can observe that the ratio of the drift to the second moment
behaves as $\frac{d-1}{2\|x\|}$; combined with the well-known fact
that the SRW is recurrent for $d=2$ and transient for $d\geq 3$,
this suggests that, to obtain recurrence, the constant in this ratio
should not be too large (in fact, at most~$\frac{1}{2}$).
Then, the second moment depends essentially on the dimension, and 
thus it is crucial to look at the drift. So,
consider a (smooth in $\R^d\setminus\{0\}$) function $g(x)=\Theta(\|x\|)$; we shall try to 
 figure out how the level sets of~$g$ should be so that 
the ``drift outside'' with respect to~$g$ ``behaves well'' (i.e.,
the drift multiplied by $\|x\|$
is uniformly bounded above by a not-so-large constant). For that,
let us look at Figure~\ref{f_level_drift}: level sets of~$g$
are indicated by solid lines, vectors' sizes correspond to transition 
probabilities.
\begin{figure}
 \centering
\includegraphics{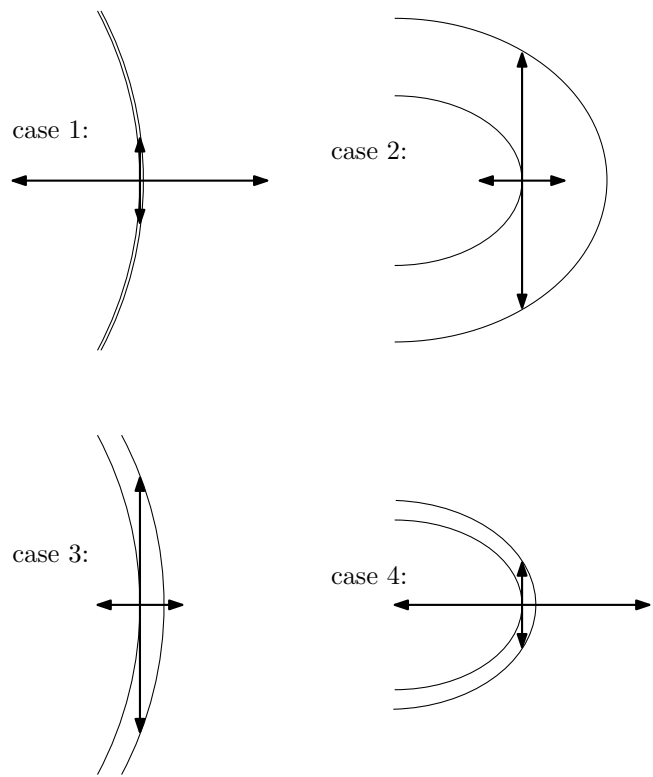}
\caption{Looking at the level sets: how large is the drift? We have \emph{very small} drift in case 1, \emph{very large} drift in case 2, and \emph{moderate}
drifts in cases 3 and 4.}
\label{f_level_drift}
\end{figure}
Then, it is intuitively clear that the case of ``moderate'' drift 
corresponds to the following:
\begin{itemize}
 \item  the ``preferred'' direction is radial, the curvature of level lines
is large, or
 \item  the ``preferred'' direction is transversal and the curvature 
of level lines is small;
\end{itemize}
also, it is clear that ``very flat'' level lines always generate small drift.
However, one cannot hope to make the level lines very flat everywhere,
as they should go around the origin. So, the idea is to find in which
places one can afford ``more curved'' level lines.
 
Observe that, for the random walk we are considering now,
the preferred direction near the axes is the radial one, while
in the ``diagonal'' regions it is in some intermediate position between transversal and radial. This indicates that 
the level sets of the Lyapunov function should look as depicted
on Figure~\ref{f_constr_f}: more curved near the axes, and more flat off the axes.
\begin{figure}
 \centering
\includegraphics{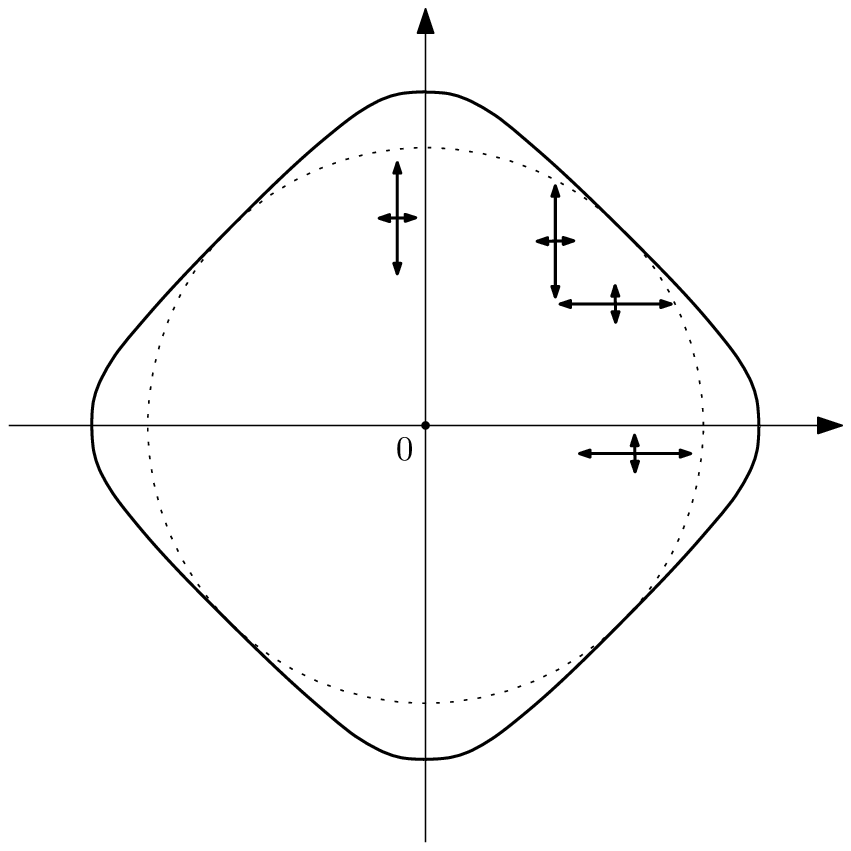}
\caption{How the level sets of~$f$ should look like?}
\label{f_constr_f}
\end{figure}

We are going to use the Lyapunov function
\[
 f(x) = \phi\Big(\frac{x}{\|x\|}\Big)\|x\|^\alpha,
\]
where $\alpha$ is a positive constant and 
$\phi:\S^{d-1}\mapsto\R$ is a positive continuous function,
symmetric in the sense that for any
$(u_0,\ldots,u_{d-1})\in\S^{d-1}$ we have
$\phi(u_0,\ldots,u_{d-1})
=\phi(\tau_0 u_{\sigma(0)},\ldots,\tau_{d-1} u_{\sigma(d-1)})$
for any permutation~$\sigma$ and any $\tau\in\{-1,1\}^d$. 
By the previous discussion, to have the level sets as
on Figure~\ref{f_constr_f}, we are aiming at constructing~$\phi$
with values close to~$1$ near the ``diagonals'' and less than~$1$
near the axes.

By symmetry, it is enough to define the function~$\phi$ 
for $u\in \S^{d-1}$ such that $u_0\geq u_{1,\ldots,d-1}\geq 0$
(clearly, it then holds that $u_0>0$), and, again by symmetry,
 it is enough to prove~\eqref{supermart} for all large
enough $x\in\Z^d$ of the same kind. 
For such $u\in \S^{d-1}$ abbreviate $s_j=u_j/u_0$, $j=1,\ldots,d-1$; 
observe that, if 
$u=x/\|x\|$, then $s_j=x_j/x_0$. We are going to look for 
the function
(for $u$ as above) $\phi(u)=1-\alpha\psi(s_1,\ldots,s_{d-1})$, 
where~$\psi$
is a function with continuous third partial derivatives 
on $[0,1]^{d-1}$ (in fact, it will become clear that the function~$\psi$
extended by means of symmetry on $[-1,1]^d$ has continuous third derivatives 
on $[-1,1]^d$; this will imply that $o$-s in the computations below are uniform).

Next, we proceed in the following way: we do calculations
in order to figure out, which conditions the function~$\psi$
should satisfy in order to guarantee that~\eqref{supermart} holds,
and then try to construct a concrete example of~$\psi$ that
satisfies these conditions.

First of all, a straightforward calculation shows that for 
any $e\in\Z^d$ with $\|e\|=1$ we have
\begin{equation}
\label{norm_Taylor}
 \|x+e\|^\alpha = \|x\|^\alpha \Big(1+\alpha\frac{\langle x,e\rangle}
{\|x\|^2}+\frac{\alpha}{2\|x\|^2}-\frac{1}{2}\alpha(2-\alpha)
\frac{\langle x,e \rangle^2}{\|x\|^2}\cdot\frac{1}{\|x\|^2}+o(\|x\|^{-2})\Big), 
\end{equation}
as $x\to \infty$.

In the computations below, we will use the abbreviations
\begin{align*}
 \psi'_j &:= \frac{\partial\psi(s_1,\ldots,s_{d-1})}{\partial s_j}, 
   \quad j=1,\ldots,d-1,\\
 \psi''_{ij} &:= \frac{\partial^2\psi(s_1,\ldots,s_{d-1})}
 {\partial s_i\partial s_j}, 
   \quad i,j=1,\ldots,d-1.
\end{align*}
 
Let us now consider $x\in \Z^d$. 
From now on we will refer to the situation when $x_0>x_{1,\ldots,d-1}\geq 0$ as the ``non-boundary case" and $x_0=x_1=\cdots=x_m>x_{m+1}\geq\ldots \geq x_{d-1}\geq 0$ for
 some $m\geq 1$ as the ``boundary case".
Observe for the boundary case the corresponding~$s$
will be of the form 
$s=(1,\ldots,(1)_m, s_{m+1},\ldots,s_{d-1})$; here and in 
the sequel we indicate the position of the symbol in a row by
placing parentheses and putting a subscript. Also, in the situation
when only one coordinate of the vector~$s$ changes, we use the
notation of the form $\psi((\tilde{s})_j)$ for
$\psi(s_1,\ldots,s_{j-1},\tilde{s},s_{j+1},\ldots,s_{d-1})$,
possibly omitting the parentheses and the subscript when the position
is clear. 

First we deal with the non-boundary case.
\newline
Let us consider $x\in\Z^d$ such that $x_0>x_{1,\ldots,d-1}\geq 0$. 
Again using~\eqref{norm_Taylor} and observing that (recall
$s_j=x_j/x_0$) 
$\frac{x_j}{x_0-1}=s_j(1+x_0^{-1}+x_0^{-2}+o(\|x\|^{-2}))$
and $\frac{x_j}{x_0+1}=s_j(1-x_0^{-1}+x_0^{-2}+o(\|x\|^{-2}))$, we write
\begin{align}
\lefteqn{\econd{f(X_{n+1})-f(X_n)}{X_n=x}}\nonumber\\
&= -(1-\alpha\psi(s))\|x\|^\alpha + \frac{\gamma}{2(\gamma+d-1)}
 \Bigg[\Big(1-\alpha\psi\big(\fract{x_1}{x_0-1},\ldots,\fract{x_{d-1}}{x_0-1}\big)\Big)
\|x-e_0\|^\alpha  \nonumber\\
& \qquad\qquad\qquad \qquad \qquad 
+ \Big(1-\alpha\psi\big(\fract{x_1}{x_0+1},\ldots,\fract{x_{d-1}}{x_0+1}\big)\Big)
\|x+e_0\|^\alpha \Bigg]\nonumber\\
&\quad  +\frac{1}{2(\gamma+d-1)}\sum_{j=1}^{d-1} \Bigg[
\Big(1-\alpha\psi\big(\fract{x_j-1}{x_0}\big)\Big)\|x-e_j\|^\alpha
+ \Big(1-\alpha\psi\big(\fract{x_j+1}{x_0}\big)\Big)\|x+e_j\|^\alpha\Bigg]
\nonumber\\
&= \|x\|^\alpha \Bigg\{\frac{\gamma}{2(\gamma+d-1)}\Bigg[\Big(1-\alpha\psi(s)
-\alpha\sum_{j=1}^{d-1}\big(\frac{s_j}{x_0}+\frac{s_j}{x_0^2}\big)\psi'_j
-\frac{\alpha}{2}\sum_{i,j=1}^{d-1}\frac{s_is_j}{x_0^2}\psi''_{ij}
+ o(\|x\|^{-2})
\Big)\nonumber\\
&\qquad \quad \qquad\qquad \qquad \times
 \Big(1-\alpha\frac{x_0}
{\|x\|^2}+\frac{\alpha}{2\|x\|^2}-\frac{1}{2}\alpha(2-\alpha)
\frac{x_0^2}{\|x\|^2}\cdot\frac{1}{\|x\|^2}+o(\|x\|^{-2})\Big)
\nonumber\\
&\qquad \qquad \qquad + \Big(1-\alpha\psi(s)
-\alpha\sum_{j=1}^{d-1}\big(-\frac{s_j}{x_0}+\frac{s_j}{x_0^2}\big)\psi'_j
-\frac{\alpha}{2}\sum_{i,j=1}^{d-1}\frac{s_is_j}{x_0^2}\psi''_{ij}
+ o(\|x\|^{-2})
\Big)\nonumber\\
&\qquad \qquad \quad\qquad \qquad \times
 \Big(1+\alpha\frac{x_0}
{\|x\|^2}+\frac{\alpha}{2\|x\|^2}-\frac{1}{2}\alpha(2-\alpha)
\frac{x_0^2}{\|x\|^2}\cdot\frac{1}{\|x\|^2}+o(\|x\|^{-2})\Big)
\nonumber\\
&\qquad \qquad \qquad - 2(1-\alpha\psi(s))\Bigg]\nonumber\\
&\quad  +\frac{1}{2(\gamma+d-1)}\sum_{j=1}^{d-1} \Bigg[- 2(1-\alpha\psi(s))+
\Big(1-\alpha\psi(s)+\alpha x_0^{-1}\psi'_j-\frac{\alpha}{2x_0^2}\psi''_{jj}
 \Big)\nonumber\\
&\qquad \quad \quad\qquad \qquad \times
 \Big(1-\alpha\frac{x_j}
{\|x\|^2}+\frac{\alpha}{2\|x\|^2}-\frac{1}{2}\alpha(2-\alpha)
\frac{x_j^2}{\|x\|^2}\cdot\frac{1}{\|x\|^2}+o(\|x\|^{-2})\Big)
\nonumber\\
& \qquad \quad \qquad + \Big(1-\alpha\psi(s)-\alpha x_0^{-1}\psi'_j-\frac{\alpha}{2x_0^2}\psi''_{jj}
 \Big)\nonumber\\
&\qquad \quad \quad\qquad \qquad \times
 \Big(1+\alpha\frac{x_j}
{\|x\|^2}+\frac{\alpha}{2\|x\|^2}-\frac{1}{2}\alpha(2-\alpha)
\frac{x_j^2}{\|x\|^2}\cdot\frac{1}{\|x\|^2}+o(\|x\|^{-2})\Big)
\Bigg] \Bigg\}\nonumber\\
&= \alpha\|x\|^{\alpha} \Bigg\{ \frac{\gamma}{\gamma+d-1}
\Bigg[\frac{1-\alpha\psi(s)}{2\|x\|^2}
-\frac{(2-\alpha)(1-\alpha\psi(s))}{2\|x\|^2}\cdot\frac{x_0^2}{
\|x\|^2 }
-\sum_{j=1}^{d-1}\Big(\frac{s_j}{x_0^2}-
\frac{\alpha s_j}{\|x\|^2} \Big)\psi'_j\nonumber\\
&\qquad \quad \qquad\qquad\qquad\qquad\qquad\qquad\qquad\qquad 
-\frac{1}{2} \sum_{i,j=1}^{d-1}\frac{s_is_j}{x_0^2}\psi''_{ij}
+o(\|x\|^{-2})\Bigg]\nonumber\\
&\qquad \quad \qquad + \frac{1}{\gamma+d-1}
\Bigg[\frac{(d-1)(1-\alpha\psi(s))}{2\|x\|^2}
+\frac{(2-\alpha)(1-\alpha\psi(s))}{2\|x\|^2}\cdot\frac{x_0^2}{
\|x\|^2 } -\frac{(2-\alpha)(1-\alpha\psi(s))}{2\|x\|^2}
\nonumber 
\\ &\qquad \quad \qquad\qquad\qquad\qquad\qquad\qquad\qquad\qquad -\sum_{j=1}^{d-1}\frac{\alpha s_j}{\|x\|^2}\psi'_j
-
\frac{1}{2}\sum_{j=1}^{d-1}\frac{1}{x_0^2}\psi''_{jj}
 +o(\|x\|^{-2})\Bigg]\Bigg\}\nonumber\\
& = -\alpha\|x\|^{\alpha-2}\Phi(x,\psi) + \alpha\|x\|^{\alpha-2}
\big(\gamma^{-1}\Phi_1(x,\psi,\gamma,\alpha)+\alpha\Phi_2(x,\psi,\gamma,
\alpha)\big),
\label{calc_interior}
\end{align}
where $\Phi_1$ and $\Phi_2$ are uniformly bounded 
for large enough~$x$, and 
\begin{equation}
\label{Phi}
 \Phi(x,\psi) = \frac{x_0^2}{\|x\|^2 }-\frac{1}{2}
 +\frac{\|x\|^2}{x_0^2}\Big(\sum_{j=1}^{d-1}s_j\psi'_j
+\frac{1}{2}\sum_{i,j=1}^{d-1}s_is_j\psi''_{ij} \Big).
\end{equation}
The idea is then to prove that, with a suitable choice for~$\psi$,
the quantity $\Phi(x,\psi)$ will be uniformly positive
for all large enough~$x$, and then the
second term in the right-hand side of~\eqref{calc_interior} 
can be controlled by choosing large~$\gamma$ and small~$\alpha$. This
will make~\eqref{calc_interior} negative for all large~$x$.

Now, in order to obtain a simplified form
for~\eqref{Phi}, we pass to the (hyper)spherical
coordinates:
\begin{align*}
s_1&=r\cos\theta_1,\\
s_2&=r\sin\theta_1\cos\theta_2,\\
&\ldots\\
s_{d-2}&=r\sin\theta_1\ldots\sin\theta_{d-3}\cos\theta_{d-2},\\
s_{d-1}&=r\sin\theta_1\ldots\sin\theta_{d-3}\sin\theta_{d-2}.
\end{align*}
Since $\frac{\|x\|^2}{x_0^2}=1+r^2$, and (abbreviating 
$\psi'_r=\frac{\partial\psi}{\partial r}$ 
and $\psi''_{rr}=\frac{\partial^2\psi}{\partial r^2}$ )
\[
 \psi'_r=\frac{1}{r}\sum_{j=1}^{d-1}s_j\psi'_j, \qquad 
 \psi''_{rr}
=\frac{1}{r^2}\sum_{i,j=1}^{d-1}s_is_j\psi''_{ij },
\]
we have
\begin{align}
  \Phi(x,\psi) &= \frac{1}{1+r^2}-\frac{1}{2}
 +(1+r^2)\Big(r\psi'_r+\frac{r^2}{2}\psi''_{rr}\Big)\nonumber\\
 &= \frac{1+r^2}{2}\Big(\frac{1-r^2}{(1+r^2)^2}+
  \big(r^2\psi'_r\big)'_r\Big). \label{Phi_sph_coord}
\end{align}
Now, we define the function~$\psi$ (it will depend on~$r$ only,
not on $\theta_1,\ldots,\theta_{d-2}$) in the following way.
First, clearly, we need to define $\psi(r)$ for
$r\in [0,\sqrt{d-1}]$. Then, observe that
\begin{equation}
\label{int_dr}
 \int_0^{\sqrt{d-1}} \frac{1-r^2}{(1+r^2)^2}\,dr =
\frac{\sqrt{d-1}}{d}>0,
\end{equation}
so, for a suitable (small enough) $\eps_0$ we can construct a
smooth function~$h$ with the following properties 
(on the Cartesian plane with coordinates $(r,y)$, think of going
from the origin along $y=\frac{r^2}{4\eps_0^2}$ until it intersects
with $y=\frac{1-r^2}{(1+r^2)^2}$ and then modify a 
little bit the curve around the intersection point to make it smooth,
see Figure~\ref{f_h_construction}):
\begin{figure}
 \centering
\includegraphics{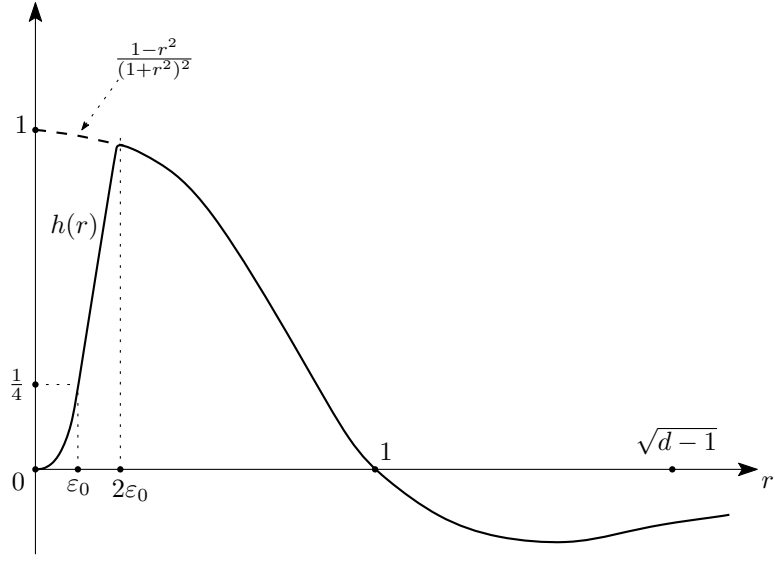}
\caption{On the construction of $h$}
\label{f_h_construction}
\end{figure}
\begin{itemize}
 \item[(i)] $0\leq h(r) \leq \frac{1-r^2}{(1+r^2)^2}$ for all~$r<2\epsilon_0$
 and $h(r)=\frac{1-r^2}{(1+r^2)^2}$ for $r\geq 2\eps_0$;
 \item[(ii)] $h(0)=0$ and $h(r)\sim \frac{r^2}{4\eps_0^2}$ as $r\to 0$;
 \item[(ii)] $\frac{1-r^2}{(1+r^2)^2}-h(r)>\frac{1}{2}$ 
for $r\leq \eps_0$;
 \item[(iv)] $b:= \int_0^{\sqrt{d-1}} h(r)\, dr > 0$ 
 (by~\eqref{int_dr} it holds in fact that~$b\in(0,1)$);
 \item[(v)] $\int_0^rh(u)\,du>\frac{br^3}{3(d-1)^{3/2}}$ 
     for all $r\in (0,\sqrt{d-1}]$. 
\end{itemize}
Denote $H(r)=\int_0^rh(u)\,du$, so that we have 
$H(\sqrt{d-1})=b$. 
Then, define for $r\in [0,\sqrt{d-1}]$
\begin{equation}
\label{df_psi}
 \psi(r) = \int_r^{\sqrt{d-1}}\Big(\frac{H(v)}{v^2}
           -\frac{bv}{3(d-1)^{3/2}}\Big) \, dv.
\end{equation}
For the function~$\psi$ defined in this way, we have 
$r^2\psi'(r)=\frac{br^3}{3(d-1)^{3/2}}-H(r)$,
so $h(r)+(r^2\psi'(r))'=b(d-1)^{-3/2}r^2$. 
By construction, it then holds that
\begin{equation}
\label{positive_Phi}
 \inf_{r\in[0,\sqrt{d-1}]}
\Big(\frac{1-r^2}{(1+r^2)^2}+\big(r^2\psi'(r)\big)'\Big) 
            \geq b(d-1)^{-3/2}\eps_0^2\wedge 
\frac{1}{2},
\end{equation}
and this (recall~\eqref{Phi} and~\eqref{Phi_sph_coord}) shows that, 
if~$\gamma$ is large enough and~$\alpha$ is small enough then 
the right-hand side of~\eqref{calc_interior} is negative for all 
large enough~$x\in\Z^d$.

To complete the proof of the theorem, it remains to deal with the boundary case.
\newline
Let $x_0=x_1=\cdots=x_m>x_{m+1}\geq\ldots \geq x_{d-1}\geq 0$ for
some $m\geq 1$. 
Using~\eqref{norm_Taylor} (up to the term of order $\|x\|^{-1}$ in the parentheses), using the fact that $\phi$ is invariant under permutations and observing that $x_0$ and $\|x\|$ are of the same order, we have
\goodbreak
\begin{align}
\lefteqn{\econd{f(X_{n+1})-f(X_n)}{X_n=x} }\nonumber\\
 &= -(1-\alpha\psi(s))\|x\|^\alpha + \frac{\gamma+m}{2(\gamma+d-1)}
 \Bigg[\Big(1-\alpha\psi\big((\fract{x_0-1}{x_0})_m\big)\Big)
\|x-e_0\|^\alpha  \nonumber\\
& \quad \quad \quad + \Big(1-\alpha\psi\big(\fract{x_0}{x_0+1},\ldots, (\fract{x_0}{x_0+1})_m, \fract{x_{m+1}}{x_0+1}\ldots,\fract{x_{d-1}}{x_0+1}\big)\Big)
\|x+e_0\|^\alpha \Bigg]\nonumber\\
&\quad + \frac{1}{2(\gamma+d-1)}\sum_{j=m+1}^{d-1} \Bigg[\Big(1-\alpha\psi\big(\fract{x_j-1}{x_0}\big)\Big)
\|x-e_j\|^\alpha + \Big(1-\alpha\psi\big(\fract{x_j+1}{x_0}\big)\Big)
\|x+e_j\|^\alpha\Bigg]\nonumber\\
&= \|x\|^\alpha \Bigg\{\frac{\gamma+m}{2(\gamma+d-1)}\Bigg[
\Big(1-\alpha\Big(\psi(s)-\frac{\psi'_m}{x_0} 
+o(\|x\|^{-1})\Big)\Big) \Big(1-\alpha\frac{x_0}{\|x\|^2}+o(\|x\|^{-1})\Big)
\nonumber\\
& \qquad+\Big(1-\alpha\Big(\psi(s)-\sum_{k=1}^m\frac{\psi'_k}{x_0}- \sum_{k=m+1}^{d-1}\frac{s_k\psi'_k}{x_0}+o(\|x\|^{-1})\Big)\Big)
% \nonumber\\
% &\qquad\qquad \times 
\Big(1+\alpha\frac{x_0}{\|x\|^2}+o(\|x\|^{-1})\Big) \nonumber\\
 & \qquad -2 (1-\alpha\psi(s))  \Bigg] \nonumber\\
&\quad + \frac{1}{2(\gamma+d-1)}\sum_{j=m+1}^{d-1}\Bigg[\Big(1-\alpha
\Big(\psi(s)-\frac{\psi'_j}{x_0} 
+o(\|x\|^{-1})\Big)\Big) \Big(1-\alpha\frac{x_j}{\|x\|^2}+o(\|x\|^{-1})\Big)
\nonumber\\
&\qquad \Big(1-\alpha\Big(\psi(s)+\frac{\psi'_j}{x_0} 
+o(\|x\|^{-1})\Big)\Big) \Big(1+\alpha\frac{x_j}{\|x\|^2}+o(\|x\|^{-1})\Big)
-2(1-\alpha\psi(s))\Bigg]
\Bigg\} \nonumber\\
&=\alpha\|x\|^\alpha\frac{\gamma+m}{2(\gamma+d-1)} \Bigg[\frac{1}{x_0}
\Big(\sum_{k=1}^{m-1}\psi'_k+2\psi'_m+\sum_{k=m+1}^{d-1}s_k\psi'_k \Big)
+o(\|x\|^{-1})\Bigg]
\label{calc_boundary}
\end{align}
(observe that in the above calculation all the terms of order 
$\|x\|^{\alpha-1}$ that correspond to the choice of coordinates 
$m+1,\ldots,d-1$ of~$x$, cancel).

Now simply note that 
by the property~(v), we have $\psi'(r)<0$ 
for all $r\in (0,\sqrt{d-1}]$.
Observe also that for some positive
constant~$\delta_0$ it holds that $\psi'(r) \leq -\delta_0$ 
for all $r\in [1,\sqrt{d-1}]$. Then (recall that in the boundary case
$s_1=1$ and $s_j\geq0$ for all $j=2,\ldots,d-1$) we have

\[
\psi'_j=\frac{s_j}{r}\psi'_r\leq 0 \text{ for all }j=1,\ldots,d-1 
\quad\text{ and }\quad
 \psi'_1(s) \leq -\frac{\delta_0}{\sqrt{d-1}}.
\]
This implies that the right-hand side of~\eqref{calc_boundary} is negative
for all large enough~$x\in\Z^d$ and thus concludes the proof of Theorem~\ref{thm:ddimrecurrence}.
\end{proof}

\section*{A conjecture}

We end this paper with an open question:

\begin{conjecture}
Let $\mu_1,\ldots,\mu_{d-1}$ be $d$-dimensional measures in $\R^d$, $d\geq 4$, with $0$ mean and $2+\beta$ moments, for some $\beta>0$, and $\ell$ an arbitrary adapted rule. Then the walk $X$ generated by these measures and the rule $\ell$ is transient.
\end{conjecture}

To answer this question, by Theorem~\ref{thm:kmeasddim} it suffices to prove the existence of a matrix $A$ satisfying the trace condition~\eqref{eq:tracecondition}. 
So far, we were able to prove it in the case when the $d-1$~covariance matrices are jointly diagonalizable.

\section*{Acknowledgements} 
We thank Itai Benjamini for asking the question that led to this work and 
 the organizers of the XV~Brazilian Probability School where this collaboration was initiated. We also thank Ronen Eldan and Miklos Racz for helpful discussions.
 The last two authors thank Microsoft Research, Redmond, and MSRI, Berkeley, where this work was completed, for their hospitality.
The work of Serguei Popov was partially
supported by CNPq (300328/2005--2) and FAPESP (2009/52379--8).

\bibliographystyle{plain}
\bibliography{biblio}

\begin{thebibliography}{10}

\bibitem{Benaim}
Michel Bena{\"{\i}}m.
\newblock Vertex-reinforced random walks and a conjecture of {P}emantle.
\newblock {\em Ann. Probab.}, 25(1):361--392, 1997.

\bibitem{BenKozScha}
Itai Benjamini, Gady Kozma, and Bruno Schapira.
\newblock A balanced excited random walk.
\newblock {\em C. R. Math. Acad. Sci. Paris}, 349(7-8):459--462, 2011.

\bibitem{BenWilson}
Itai Benjamini and David~B. Wilson.
\newblock Excited random walk.
\newblock {\em Electron. Comm. Probab.}, 8:86--92 (electronic), 2003.

\bibitem{BerRam}
Jean B{\'e}rard and Alejandro Ram{\'{\i}}rez.
\newblock Central limit theorem for the excited random walk in dimension
  {$D\geq 2$}.
\newblock {\em Electron. Comm. Probab.}, 12:303--314 (electronic), 2007.

\bibitem{Esseen}
C.~G. Esseen.
\newblock On the concentration function of a sum of independent random
  variables.
\newblock {\em Z. Wahrscheinlichkeitstheorie und Verw. Gebiete}, 9:290--308,
  1968.

\bibitem{FayMalMensh}
G.~Fayolle, V.~A. Malyshev, and M.~V. Menshikov.
\newblock {\em Topics in the constructive theory of countable {M}arkov chains}.
\newblock Cambridge University Press, Cambridge, 1995.

\bibitem{HornJohnson}
Roger~A. Horn and Charles~R. Johnson.
\newblock {\em Matrix analysis}.
\newblock Cambridge University Press, Cambridge, 1990.
\newblock Corrected reprint of the 1985 original.

\bibitem{MenPopRamVach}
M.~{Menshikov}, S.~{Popov}, A.~{Ramirez}, and M.~{Vachkovskaia}.
\newblock {On a general many-dimensional excited random walk}.
\newblock {\em ArXiv e-prints}.
\newblock {To appear in: \textit{Ann. Probab.}}

\bibitem{MerkRolSil}
Franz Merkl and Silke W.~W. Rolles.
\newblock Recurrence of edge-reinforced random walk on a two-dimensional graph.
\newblock {\em Ann. Probab.}, 37(5):1679--1714, 2009.

\bibitem{PemVolkov}
Robin Pemantle and Stanislav Volkov.
\newblock Vertex-reinforced random walk on {${\bf Z}$} has finite range.
\newblock {\em Ann. Probab.}, 27(3):1368--1388, 1999.

\bibitem{RaimSchap}
Olivier Raimond and Bruno Schapira.
\newblock On some generalized reinforced random walk on integers.
\newblock {\em Electron. J. Probab.}, 14:no. 60, 1770--1789, 2009.

\bibitem{HofHol}
Remco van~der Hofstad and Mark Holmes.
\newblock Monotonicity for excited random walk in high dimensions.
\newblock {\em Probab. Theory Related Fields}, 147(1-2):333--348, 2010.

\bibitem{Volkov}
Stanislav Volkov.
\newblock Vertex-reinforced random walk on arbitrary graphs.
\newblock {\em Ann. Probab.}, 29(1):66--91, 2001.

\bibitem{Zerner}
Martin P.~W. Zerner.
\newblock Recurrence and transience of excited random walks on {$\mathbb{Z}^d$}
  and strips.
\newblock {\em Electron. Comm. Probab.}, 11:118--128 (electronic), 2006.

\end{thebibliography}

\end{document}